\documentclass[preprint,12pt]{amsart}

\usepackage[margin=1.5in]{geometry}

\usepackage{amsmath}
\usepackage{amsfonts}
\usepackage{amssymb}
\usepackage{xypic}
\usepackage{longtable}
\usepackage{amsthm}
\usepackage{booktabs}
\usepackage{caption}

\newlength{\thickarrayrulewidth}
\newcommand{\Aut}{\mathrm{Aut}}
\newcommand{\Div}{\mathrm{Div}}

\theoremstyle{plain}
\newtheorem{theorem}{Theorem}[section]

\newtheorem{lemma}[theorem]{Lemma}
\newtheorem{proposition}[theorem]{Proposition}
\theoremstyle{definition}
\newtheorem{definition}[theorem]{Definition}

\newtheorem{rem}[theorem]{Remark}

\makeatletter
\def\ps@pprintTitle{%
  \let\@oddhead\@empty
  \let\@evenhead\@empty
  \let\@oddfoot\@empty
  \let\@evenfoot\@oddfoot
}
\makeatother

\title{Galois subspaces for smooth projective curves}
\author{Robert Auffarth and Sebasti\'an Rahausen}
\address{R. Auffarth \\Departamento de Matem\'aticas, Facultad de
Ciencias, Universidad de Chile, Santiago\\Chile}
\email{rfauffar@uchile.cl}
\address{S. Rahausen \\Facultad de Matem\'aticas, Pontificia Universidad Cat\'olica de Chile, Santiago\\Chile}
\email{srahausen@gmail.com}


\begin{document}

\maketitle

\begin{abstract}
Given an embedding of a smooth projective curve $X$ of genus $g\geq1$ into $\mathbb{P}^N$, we study the locus of linear subspaces of $\mathbb{P}^N$ of codimension 2 such that projection from said subspace, composed with the embedding, gives a Galois morphism $X\to\mathbb{P}^1$. For genus $g\geq2$ we prove that this locus is a smooth projective variety with components isomorphic to projective spaces. If $g=1$ and the embedding is given by a complete linear system, we prove that this locus is also a smooth projective variety whose positive-dimensional components are isomorphic to projective bundles over \'etale quotients of the elliptic curve, and we describe these components explicitly. 
\end{abstract}

\section{Introduction}

Let $X$ be a smooth projective curve over an algebraically closed field $k$ of characteristic 0, and let $\varphi:X\hookrightarrow\mathbb{P}^N$ be an embedding. If $W\in\mathbb{G}(N-2,N)$ is a linear subspace, let $\pi_W:\mathbb{P}^N\dashrightarrow\mathbb{P}^1$ denote projection from $W$. We are interested in studying the locus
\[G_{X,\varphi}:=\{W\in\mathbb{G}(N-2,N):\pi_W\circ\varphi\text{ is Galois}\}.\]
Any $W\in G_{X,\varphi}$ will be called a \textit{Galois subspace} for $\varphi$. 

This set has been widely studied in many cases in the context of \textit{Galois embeddings}, a notion introduced by Yoshihara in \cite{Yoshi}, and studied in detail in articles such as \cite{FT}, \cite{Yoshi3}, \cite{FM}, \cite{DM}, and \cite{Takahashi}, and in particular in the genus one case in \cite{DY}, \cite{KY}, \cite{KY2} and \cite{Yoshi4}. We would also like to point out the list of open problems on this topic that is maintained by Fukasawa and Yoshihara \cite{Yoshiproblems}. We note that in these articles special attention has been paid to those Galois subspaces that are disjoint from $\varphi(X)$, which in this article we will call \textit{disjoint Galois subspaces} for $\varphi$.

The purpose of this article is to give a general structure theorem for $G_{X,\varphi}$ in the case that the genus of $g$ is greater than 0. The genus 0 case was studied in the authors' previous article \cite{AR}, in which $G_{X,\varphi}$ was shown to be a locally closed subvariety of the Grassmannian when $\varphi$ is the Veronese embedding. In particular the subvariety of $G_{X,\varphi}$ consisting of disjoint Galois subspaces was shown to be positive-dimensional and non-projective. In this article, on the other hand, we prove the following:

\begin{theorem}\label{1st theorem} Let $X$ be a curve of genus $g\geq1$ and let $\varphi:X\hookrightarrow\mathbb{P}^N$ be an embedding. Then $G_{X,\varphi}$ is a projective variety. Moreover,

\begin{enumerate}
\item if $g=1$ and $\varphi$ is given by a complete linear system, then $G_{X,\varphi}$ is non-empty, smooth, and for $N\geq 3$ it is reducible. In fact, its 0-dimensional components are exactly the set of disjoint Galois subspaces and its positive-dimensional components are projective bundles over \'etale quot\-ients of $E$.

\item\label{ggeq1} if $g\geq2$ and $G_{X\varphi}\neq\varnothing$, then the components of $G_{X,\varphi}$ are points or projective spaces.
\end{enumerate}
\end{theorem}

For $g\geq2$, not much more can be said with this kind of generality. However, in the case of elliptic curves we can be much more explicit, and this is the case that will occupy most of the article. In the case of a general elliptic curve we can say the following:

\begin{theorem}
Let $E$ be an elliptic curve with $j$-invariant different from $0$ and $1728$, let $\varphi:E\hookrightarrow\mathbb{P}^N$ be an embedding given by a complete linear system, and let $\sigma:\mathbb{N}\to\mathbb{N}$ be the function that adds the positive divisors of an integer. Then we have the following:
\begin{enumerate}
\item $G_{E,\varphi}$ has $\left(\frac{N+1}{2}\right)\sigma\left(\frac{N+1}{2}\right)$ 0-dimensional components if $N$ is odd and 0 if $N$ is even. 
\item If $2\leq s\leq N$, $G_{E,\varphi}$ has $\sigma\left(\frac{s}{2}\right)$ components of dimension $N+1-s$ if $s$ is even and 0 if $s$ is odd. Moreover each subspace belonging to any of these components has an associated Galois group of order $s$.
\item If $G$ is the Galois group of a non-disjoint Galois subspace $W$ for $\varphi$ and $H$ is the translation subgroup of $G$, then the component of $G_{E,\varphi}$ that contains $W$ is a $\mathbb{P}^{N-|G|}$-bundle over $E/H$.
\end{enumerate}
\end{theorem}

Here we present the general case as it is easier and cleaner to state, but we know exactly what happens when $j=0$ and $j=1728$ as well. Indeed, Theorem \ref{nondisjoint} gives the general result for any $j$-invariant. Moreover, after the statement of Theorem \ref{nondisjoint} we give several tables that show how many components $G_{E,\varphi}$ has for small values of $N$. We remark that these values coincide with the number of Galois subspaces given in \cite{DY}, \cite{KY2} and \cite{Yoshi4} for elliptic curves embedded in $\mathbb{P}^3$. 

The explicit calculation of the number of components depends on knowing how many $\xi$-stable subgroups of $E[m]$ there are of order $m$, where $m\in\mathbb{Z}_{>0}$ and $\xi$ is a non-trivial automorphism of $E$ that fixes the origin. An explicit expression for this number is the subject of the appendix at the end of this article.

\noindent\textit{Acknowledgements:} The first author was partially supported by ANID-Fondecyt Grant 11180965, and the second author was partially supported by the ANID scholarship Doctorado Nacional 2020.

\section{Some generalities on linear systems on curves}

Let $X$ be an irreducible smooth projective curve, and for $i=1,2$ let $\varphi_i:X\to\mathbb{P}^{N_i}$ be a morphism induced by a base-point-free linear system $V_i\leq H^0(X,\mathcal{L}_i)$ for a line bundle $\mathcal{L}_i$. In this section we will study the projective variety
\[\frak{X}_{\varphi_1,\varphi_2}:=\{[\phi]\in\mathbb{P}\mathrm{Hom}(V_1^\vee,V_2^\vee):[\phi]\varphi_1=\varphi_2\text{ when defined}\},\]
where we can see $[\phi]$ as the induced rational map between $\mathbb{P}(V_1)^\vee$ and $\mathbb{P}(V_2)^\vee$. Note that we have an embedding 
\[\chi_{\varphi_1,\varphi_2}:\frak{X}_{\varphi_1,\varphi_2}\to\mathbb{G}(m,\mathbb{P}(V_1)^\vee),\]
where $m=\dim V_1-\dim V_2-1$, that sends $[\phi]$ to $\mathbb{P}(\ker\phi)$. Note moreover that for every automorphism $\rho\in\mathrm{Aut}(\mathbb{P}(V_2)^\vee)$, we obtain an isomorphism $\rho_{\varphi_1,\varphi_2}:\frak{X}_{\varphi_1,\varphi_2}\to\frak{X}_{\varphi_1,\rho\circ\varphi_2}$ where $[\phi]\mapsto[\rho\circ\phi]$ that makes the following diagram commute: 

\centerline{\xymatrix{
\frak{X}_{\varphi_1,\varphi_2}\ar[dr]_{\chi_{\varphi_1,\varphi_2}}\ar[rr]^{\rho_{\varphi_1,\varphi_2}}&&\frak{X}_{\varphi_1,\rho\circ\varphi_2}\ar[dl]^{\chi_{\varphi_1,\rho\circ\varphi_2}}\\
&\mathbb{G}(m,\mathbb{P}(V_1)^\vee)&}}

This implies that we can think of $\frak{X}_{\varphi_1,\varphi_2}$ as the variety of all subspaces $W\in\mathbb{G}(m,\mathbb{P}(V_1)^{\vee})$ such that $\pi_W\circ\varphi_1$ coincides with $\varphi_2$ up to an automorphism of $\mathbb{P}(V_2)^\vee$, where $\pi_W:\mathbb{P}(V_1)^\vee\to\mathbb{P}(V_2)^\vee$ is projection from $\mathbb{P}(V_1)^{\vee}$ with center $W$ (we can fix an embedding of $\mathbb{P}(V_2)^\vee$ into $\mathbb{P}(V_1)^\vee$ beforehand so that the linear projection makes sense).\\

We will now study the structure of $\frak{X}_{\varphi_1,\varphi_2}$. Let $[\phi]\in \frak{X}_{\varphi_1,\varphi_2}$. Clearly such a $\phi$ comes from a linear map $T_\phi:V_2\to V_1$. Moreover, since $[\phi]\varphi_1=\varphi_2$, we have that if $H_{i,p}:=\{s\in V_i:s(p)=0\}$ (which is just equal to $\varphi_i(p)$), then
\begin{equation}\label{commute}T_\phi^{-1}(H_{1,p})=H_{2,p},\end{equation}
as long as the image of $T_\phi$ is not contained in $H_{1,p}$. In any case, we always have that $T_\phi(H_{2,p})\subseteq H_{1,p}$.

\begin{lemma}
There exists $f\in H^0(X,\mathcal{L}_1\otimes\mathcal{L}_2^{-1})$ such that $T_\phi(s)=fs$. 
\end{lemma}
\begin{proof}
We first observe that for all $s\in V_2$, $f_s:=T_\phi(s)/s$ is a global rational section of the line bundle $\mathcal{L}_1\otimes\mathcal{L}_2^{-1}$. Let $s$ be a section in $V_2$ such that $\{s=0\}$ is reduced (such an $s$ always exists by Bertini's Theorem). In particular, if $p\in\{s=0\}$, then by (\ref{commute}), we have that $T_\phi(s)(p)=0$. Since $\text{ord}_p(s)=1$, we have that $T_\phi(s)/s$ is regular at $p$. Since $p$ is arbitrary, we have that $f_s$ is regular.

Now if $s_1,s_2\in V_2$ are such that $\{s_i=0\}$ is reduced for $i=1,2$, we have that for $p\in X$, the section $s_2(p)s_1-s_1(p)s_2\in\langle s_1,s_2\rangle$ vanishes at $p$. Note that for this specific calculation we are seeing $s_1$ and $s_2$ as rational functions that are regular at $p$, which can always be done. We have that
\[T_\phi(s_2(p)s_1-s_1(p)s_2)=s_2(p) s_1 f_{s_1}-s_1(p) s_2 f_{s_2}\]
and by (\ref{commute}), it also vanishes at $p$. This implies that if $s_1(p)s_2(p)\neq0$, then $f_{s_1}(p)=f_{s_2}(p)$. Therefore on a non-empty open set $f_{s_1}=f_{s_2}$, and thus they are equal everywhere. 

This implies that on a non-empty open set (i.e. the set of all $s\in V_2$ such that $\{s=0\}$ is reduced), $T_\phi$ is multiplication by an $f\in H^0(X,\mathcal{L}_1\otimes\mathcal{L}_2^{-1})$. Therefore this must occur on all $V_2$.
\end{proof}

Reciprocally, we note that if $f\in H^0(X,\mathcal{L}_1\otimes\mathcal{L}_2^{-1})$ is such that $fV_2\subseteq V_1$, we obtain the linear transformation $\lambda_f:V_2\to V_1$ which is just multiplication by $f$. Moreover, if $fV_2\not\subseteq H_{1,p}$ (which occurs precisely for $p\notin\{f=0\}$), then $fV_2\cap H_{1,p}=H_{2,p}$, which implies that $[\lambda_f]\varphi_1=\varphi_2$.

\begin{lemma}
For $f,g\in H^0(X,\mathcal{L}_1\otimes\mathcal{L}_2^{-1})$, we have that $fV_2=gV_2$ if and only if $f$ is a constant multiple of $g$.
\end{lemma}
\begin{proof}
We have that $f/g$ is a rational function on $X$ such that 
\[(f/g)V_2=V_2.\] 
Now if $f/g$ has a zero, then it would be a base point of $V_2$, a contradiction. Therefore $f/g$ is constant.
\end{proof}

We now come to the main theorem of this section:

\begin{theorem}\label{linear systems} If $\mathcal{L}_1$ and $\mathcal{L}_2$ are line bundles on $X$ and for $i=1,2$ $V_i$ is a base-point-free linear system in $H^0(X,\mathcal{L}_i)$, then the map $[f]\mapsto[\lambda_f^\vee]$ gives an isomorphism
\[\{[f]\in\mathbb{P}H^0(X,\mathcal{L}_1\otimes\mathcal{L}_2^{-1}):fV_2\subseteq V_1\}\simeq\frak{X}_{\varphi_1,\varphi_2}\] 
In particular, if $\varphi_1$ is the morphism associated to the complete linear system $|\mathcal{L}_1|$, we have that
\[\frak{X}_{\varphi_1,\varphi_2}\simeq\mathbb{P}H^0(X,\mathcal{L}_1\otimes\mathcal{L}_2^{-1}).\]
\end{theorem}
\begin{proof}
It is clear by our previous analysis that the above correspondence is a bijective morphism. The only thing that is left to prove is that the differential at every point is injective. We leave this to the reader as a fun and interesting exercise.
\end{proof}

\section{Galois subspaces for $g\geq 2$}
Since the genus 1 case is more involved, we will prove Part (\ref{ggeq1}) of Theorem \ref{1st theorem} first. 

Let $X$ be a smooth projective curve of genus $g\geq 2$, let $G\leq\mathrm{Aut}(X)$ be such that $X/G\simeq\mathbb{P}^1$, and let $\varphi:X\hookrightarrow\mathbb{P}^N$ be an embedding. If $\pi_G:X\to X/G\simeq\mathbb{P}^1$ is the natural projection, then Theorem \ref{linear systems} assures us that
\[\frak{X}_{\varphi,\pi_G}\simeq\{[f]\in\mathbb{P}H^0(X,\mathcal{L}_1\otimes\mathcal{L}_2^{-1}):fH^0(X,\mathcal{L}_2)^G\subseteq V\},\]
where $\mathcal{L}_1=\varphi^*\mathcal{O}_{\mathbb{P}^N}(1)$, $\mathcal{L}_2=\pi_G^*\mathcal{O}_{\mathbb{P}^1}(1)$ and $V\leq H^0(X,\mathcal{L}_1)$ is the linear system that induces $\varphi$. This is either empty or isomorphic to projective space, and corresponds to the subvariety of $\mathbb{G}(N-2,N)$ that consists of subspaces such that projection from said subspace, composed with $\varphi$, is Galois with Galois group $G$. 

Since $\mathrm{Aut}(X)$ is finite, then
\[G_{X,\varphi}\simeq\bigsqcup_{\stackrel{G\leq\mathrm{Aut}(X)}{X/G\simeq\mathbb{P}^1}}\frak{X}_{\varphi,\pi_G}\]
is either empty or a finite disjoint union of points and/or projective spaces. This proves Part (\ref{ggeq1}) of Theorem \ref{1st theorem}.

\begin{rem} We note that the disjoint Galois subspaces are contained in the 0-dimensional components of $G_{X,\varphi}$. This follows immediately from \cite[Corollary 2.6]{Yoshi} (for Yoshihara a Galois subspace is by definition what we call a disjoint Galois subspace).
\end{rem}

\section{Galois subspaces for $g=1$}\label{Section 1}

For the rest of the article we will concentrate on the case when $X$ is an elliptic curve. For this reason, we will denote by $E$ the elliptic curve to avoid confusion with the previous case.

If $E$ is an elliptic curve, let $\text{Aut}_0(E):=\{g\in\text{Aut}(E):g(0)=0\}$ be the group of automorphisms that fix the origin. Let $G\leq\Aut(E)$ be a finite subgroup (that does not necessarily fix the origin) such that the quotient curve $E/G$ is isomorphic to $\mathbb{P}^1$. We have the following exact sequence of groups
\[0\to T_G\to G\to G_0\to 1\]
where $T_G$ is the subgroup of $G$ consisting of translations, $G_0\leq\text{Aut}_0(E)$ is the group $\{g-g(0):g\in G\}$, and the homomorphism $G\to G_0$ is given by $g\mapsto g-g(0)$. Now we have that $G_0$ is a cyclic group of order $\ell\in\{2,3,4,6\}$ and $T_G$ can be identified with a subgroup of
\[E[m]\simeq\mathbb{Z}/m\mathbb{Z}\times\mathbb{Z}/m\mathbb{Z}\] 
where $m=|T_G|=|G|/\ell$. If $x\in E$, we will write $t_x$ for the automorphism translation by $x$. For any subgroup $H\leq E$, we will identify $H$ with the group $\{t_h:h\in H\}$; this identification should not cause confusion.  

Therefore, every finite group of automorphisms of an elliptic curve that does not consist solely of translations is of the form
\[G_{H,\xi,q}:=\langle \mu_{\xi,q},H\rangle\]
where $\mu_{\xi,q}(x)=\xi x+q$, $\xi$ is a an automorphism of $E$ that fixes the origin that we will identify with a primitive $\ell$th root of unity for $\ell\in\{2,3,4,6\}$, $q\in E$ and $H$ is a finite subgroup of $E$. Moreover, we have that $H$ is $\xi$-stable, in the sense that $\xi H=H$. Since $H$ is normal in $G_{H,\xi,q}$, we have that every element in $G_{H,\xi,q}$ can be written as $\mu_{\xi,q}^it_h$ for some $0\leq i\leq \ell-1$ and $h\in H$.\\

We will begin our study of elliptic curves by considering the embedding $\varphi_D:E\hookrightarrow|D|^\vee\simeq\mathbb{P}^N$ given by a complete linear system $H^0(E,\mathcal{O}_E(D))$ for a very ample divisor $D$. By the Riemann-Roch Theorem, we have that $\deg(D)=N+1$. Throughout the following sections we will work mostly with the integer $n:=\deg(D)$ instead of $N$.

We will first classify all Galois subspaces for $\varphi_D$ that are disjoint from $\varphi_D(E)$. We recall (see \cite[Theorem 2.2]{Yoshi}) that these disjoint Galois subspaces are in 1-1 correspondence with those subgroups $G\leq\mathrm{Aut}(E)$ such that
\begin{enumerate}
\item[1.] $|G|=\deg(D)$
\item[2.] $E/G\simeq\mathbb{P}^1$
\item[3.] $\pi^*\mathcal{O}_{\mathbb{P}^1}(1)\simeq\mathcal{O}_E(D)$ where $\pi:E\to E/G$ is the natural projection.
\end{enumerate} 

Let $\mathcal{S}:\Div(E)\to E$ be the addition function
\[\sum_{p}n_p[p]\mapsto\sum_{p}n_pp,\]
where in order to avoid confusion, we write points in a divisor with brackets. We recall that the Abel-Jacobi Theorem for elliptic curves states that if $D_1,D_2\in\Div(E)$, then $D_1$ and $D_2$ are linearly equivalent if and only if $\deg D_1=\deg D_2$ and $\mathcal{S}(D_1)=\mathcal{S}(D_2)$.

Let $\xi\in\text{Aut}_0(E)$ be an automorphism of order $\ell\geq2$, and for $m\geq1$ consider the endomorphism of $E$
\[\varepsilon_{\xi,m}(x):=\left\{\begin{array}{ll}m x&\text{if } \ell=2\\m(2+\xi)x&\text{if }\ell=3\\2m(1+\xi)x&\text{if }\ell=4\\6m\xi x&\text{if }\ell=6\end{array}\right.
\]
and the divisor
\[D_{H,\xi,q}:=\sum_{g\in G_{H,\xi,q}}[g(0)].\]

\begin{lemma}\label{endoepsilon}
The divisor $D_{H,\xi,q}$ is of degree $|G_{H,\xi,q}|$, it is fixed by $G_{H,\xi,q}$ (as a divisor, not point-wise), and
\[\mathcal{S}(D_{H,\xi,q})=\varepsilon_{\xi,|H|}(q).\]
\end{lemma}
\begin{proof}
The first two statements are clear since if $\pi:E\to E/G_{H,\xi,q}$ is the natural projection, $D_{H,\xi,q}=\pi^*([\pi(0)])$. For the last equality, we have that
\begin{eqnarray}\nonumber\mathcal{S}(D_{H,\xi,q})=\sum_{i=0}^{\ell-1}\sum_{h\in H}\mu_{\xi,q}^it_h(0)&=&\sum_{i=0}^{\ell-1}\sum_{h\in H}\left(\xi^ih+\frac{\xi^i-1}{\xi-1}q\right)\\
\nonumber&=&\sum_{h\in H}\sum_{i=0}^{\ell-1}\frac{\xi^i-1}{\xi-1}q.
\end{eqnarray}
Here we are abusing the notation, since $\frac{\xi^i-1}{\xi-1}$ is actually a polynomial in $\xi$, and thus an endomorphism of $E$. By reviewing every case according to the value of $\ell$, we obtain the equality.
\end{proof}

\begin{proposition}\label{galois groups}
Let $H\leq E$ be a finite subgroup of order $m$, let $\xi\in\Aut_0(E)$ be an automorphism of order $\ell\geq2$, let $q\in E$ be any point, and let $D$ be a very ample divisor of degree $n=\ell m\geq3$. Then the set of Galois groups $G$ with $T_G=H$ and $G_0=\langle\xi\rangle$ that are associated with disjoint Galois subspaces for the embedding $\varphi_D$ is exactly 
\[\{G_{H,\xi,q}:q\in\varepsilon_{\xi,m}^{-1}(\mathcal{S}(D))\}.\]
\end{proposition}
\begin{proof}
Let $s:=\mathcal{S}(D)$ be the sum of the elements of $D$. Then if we choose $q\in\varepsilon_{\xi,m}^{-1}(s)$, we obtain that
\[D_{H,\xi,q}\sim D.\]

Reciprocally, let $G_{H,\xi,q}$ be a Galois group for $\varphi_D$; we will show that $q\in\varepsilon_{\xi,m}^{-1}(s)$. We note that if $\pi:E\to E/G_{H,\xi,q}$ is the quotient map, then for every $t\in\mathbb{P}^1$, $\pi^*([t])\in|D|$. Now $\pi^*([t])=\sum_{g\in G_{H,\xi,q}}[g(p)]$ for any $p\in\pi^{-1}(t)$. A simple check shows that
\[\mathcal{S}(\pi^*([t]))=\varepsilon_{\xi,m}(q)\]
and is therefore independent of $p$. This sum must also be equal to $\mathcal{S}(D)$ since the divisors are linearly equivalent, and we are finished.
\end{proof}

We now want to understand when two of these groups can be equal. Assume that $G_{H,\xi,q}=G_{H',\xi',q'}$. We first observe that $H=H'$ since the translation subgroup is unique. Next, we must have that $\xi=(\xi')^i$ for some $i$ coprime with $\ell$. But this implies that, after possibly changing $q'$, we can assume that $\xi=\xi'$. Therefore we now want to know when $G_{H,\xi,q}=G_{H,\xi,q'}$.

\begin{lemma}\label{H action}
We have that $G_{H,\xi,q}=G_{H,\xi,q'}$ if and only if $q-q'\in H$.
\end{lemma}
\begin{proof}
Obviously if $q-q'\in H$, then $G_{H,\xi,q}=G_{H,\xi,q'}$. For the other direction, we must have that $\mu_{\xi,q}=\mu_{\xi,q'}^i+h$ for some $i$ and some $h\in H$. By equaling the linear parts of these automorphisms, we must have then that $i=1$, and therefore $q=q'+h$.
\end{proof}

This lemma gives us a way to count how many Galois groups with a fixed $H$ appear for a given Galois embedding.

\begin{proposition}
Let $D\in\Div(E)$ be a divisor of degree $n\geq 3$, let $n=\ell m$ for $\ell\in\{2,3,4,6\}$, let $\xi$ be an automorphism of order $\ell$ of $E$ that fixes the origin and let $H\subseteq E[m]$ be a $\xi$-stable subgroup of order $m$. Then there are exactly $\deg\varepsilon_{\xi,m}/m$ different Galois groups associated to disjoint Galois subspaces for the embedding $\varphi_D$ with translation group equal to $H$. 
\end{proposition}
\begin{proof}
This follows immediately from Theorem \ref{galois groups} and Lemma \ref{H action}, since we only need to count the amount of elements in the quotient space
\[\varepsilon_{\xi,m}^{-1}(\mathcal{S}(D))/H.\]
This has $\deg\varepsilon_{\xi,m}/m$ points since $H$ acts freely.
\end{proof}

In order to count the total number of Galois subspaces, we will introduce the following notation:

\begin{definition}
For $\ell\in\{2,3,4,6\}$, we define the function
\[\psi_{\ell}:\mathcal{M}_{1,1}\times\mathbb{Q}_{>0}\to\mathbb{N}\cup\{0\}\]
as follows:
\begin{enumerate}
\item If $m\in\mathbb{N}$ and $E$ is an elliptic curve that possesses an automorphism $\xi\in\text{Aut}_0(E)$ of order $\ell$, then $\psi_{\ell}(E,m)$ is the number of $\xi$-stable subgroups of $E[m]$ of order $m$.
\item In all other cases, $\psi_{\ell}(E,m)=0$.
\end{enumerate}
\end{definition}

Here $\mathcal{M}_{1,1}$ denotes the moduli space of elliptic curves over $k$. In Theorem \ref{appendixtheorem} of the appendix, we will calculate this number explicitly. As a corollary to the previous proposition, we obtain:

\begin{theorem}\label{maintheorem}
Let $E$ be an elliptic curve, and let $D\in\text{Div}(E)$ be a very ample divisor of degree $n\geq3$. Then the number of disjoint Galois subspaces for $\varphi_D$ is
\[\frac{n}{2}\psi_{2}\left(E,\frac{n}{2}\right)+n\psi_{3}\left(E,\frac{n}{3}\right)+2n\psi_{4}\left(E,\frac{n}{4}\right)+6n\psi_{6}\left(E,\frac{n}{6}\right).\]

\end{theorem}

\section{Non-disjoint Galois subspaces for elliptic curves}\label{nd elliptic curves}

Let $E$ be an elliptic curve, $D$ a very ample divisor on $E$ of degree $n$, and now we will take $G\leq\mathrm{Aut}_0(E)$ to be a finite group of order $|G|<n$. In this case any linear projection $|D|^\vee\dashrightarrow\mathbb{P}^1$ such that upon composing with $\varphi_D$ we obtain a Galois morphism with Galois group $G$ is therefore centered at a codimension 2 subspace of $|D|^\vee$ that is not disjoint from $\varphi_D(E)$. 

By the previous section we have that $G=G_{H,\xi,q}$ for a certain subgroup $H\leq E$, $\xi\in\mathrm{Aut}_0(E)$ a non-trivial automorphism of order $\ell$ that fixes the origin and $q\in E$. The quotient map $\pi_G:E\to E/G\simeq\mathbb{P}^1$ is given by the linear system $H^0(E,\mathcal{O}_E(D_{H,\xi,q}))^G$. 

Theorem \ref{linear systems} assures us then that the locus of all $W\in\mathbb{G}(n-3,|D|^\vee)$ such that $\pi_W\circ\varphi_D$ is Galois with Galois group $G_{H,\xi,q}$ is isomorphic to 
\[\mathbb{P}H^0(E,\mathcal{O}_E(D-D_{H,\xi,q})).\]
We recall the embedding
\[\chi_{H,\xi,q}:=\chi_{\varphi_D,\pi_{G}}:\mathbb{P}H^0(E,\mathcal{O}_E(D-D_{H,\xi,q}))\to\mathbb{G}(n-3,|D|^\vee)\]
whose image consists of Galois subspaces for $D$ with Galois group $G_{H,\xi,q}$. However this is not enough for us to describe the components of $G_{E,\varphi_D}$, since we can vary $q$ in $E$ to obtain a larger family. This is what we will do in what follows.

For $r:=\deg(D)-\ell|H|$, consider the map
\[\delta_{H,\xi}:E\to\mathrm{Pic}^{r}(E)=:J_r\]
\[q\mapsto\mathcal{O}_E(D-D_{H,\xi,q}).\] 
We see that this is indeed a morphism since by Lemma \ref{endoepsilon}, 
\[\delta_{H,\xi}(q)=\mathcal{O}_E(D-[\varepsilon_{\xi,|H|}(q)]-(\ell|H|-1)[0]).\]
Recall that we have the morphism $\mathrm{Sym}^{r}(E)\to\mathrm{Pic}^{r}(E)$ that sends an effective divisor to its associated line bundle. Let us now consider the fiber product
\[\frak{d}_{H,\xi}:=\mathrm{Sym}^{r}(E)\times_{J_r}E.\]
We observe that the second projection $\pi_2:\frak{d}_{H,\xi}\to E$ gives $\frak{d}_{H,\xi}$ the structure of a projective bundle over $E$ since for each $q\in E$, $\pi_2^{-1}(q)\simeq|D-D_{H,\xi,q}|=\mathbb{P}H^0(E,\mathcal{O}_E(D-D_{H,\xi,q}))$. We can now define the morphism
\[\Psi_{H,\xi}:\frak{d}_{H,\xi}\to\mathbb{G}(n-3,|D|^\vee)\]
\[(F,q)\mapsto\chi_{H,\xi,q}(F)\]
whose image consists of all Galois subspaces for $D$ such that its Galois group is of the form $G_{H,\xi,q}$ for some $q\in E$. We note that we are abusing the notation here somewhat, since we are sweeping under the rug the canonical isomorphism $\mathbb{P}H^0(E,\mathcal{O}_E(D-D_{H,\xi,q}))\simeq|D-D_{H,\xi,q}|$.

Note that the above discussion makes sense even if $r=0$ if we set $\frak{d}_{H,\xi}$ in this case to be $\ker(E\to\mathrm{Pic}^0(E))$ where the morphism sends $q\mapsto\mathcal{O}_E(D-D_{H,\xi,q})$. This shows that the 0-dimensional components are exactly the disjoint Galois subspaces for $\varphi_D$.

By taking the union over all $H$ and $\xi$ (including the case when $\ell|H|=\deg(D)$), we obtain that
\[G_{E,\varphi_D}=\bigsqcup_{H,\xi}\mathrm{Im}(\Psi_{H,\xi})\subseteq\mathbb{G}(n-3,|D|^\vee).\]
We note that the components of $G_{E,\varphi_D}$ do not intersect each other, since the components are in 1-1 correspondence with the pairs $(H,\langle\xi\rangle)$ where $\xi\in\Aut_0(E)$ is non-trivial, $H\leq E$ is $\xi$-stable, and $|H||\xi|<n$. Since each element of a component of $G_{E,D}$ has a uniquely determined Galois group associated to it, an element of an intersection would possess two different Galois groups, a contradiction.

We finalize the proof of Theorem \ref{1st theorem} by observing that, by Lemma \ref{H action}, $H$ acts on $\frak{d}_{H,\xi}$ on the second coordinate, and the morphism $\Psi_{H,\xi}$ descends to an embedding
\[\overline{\Psi}_{H,\xi}:\frak{d}_{H,\xi}/H\to G_{E,\varphi_D}.\]
Therefore every component of $G_{E,D}$ is a projective bundle over a finite \'etale quotient of $E$.

\begin{rem}
When $\ell\neq 2$, we observe that since $H$ is $\xi$-stable, then $\xi$ descends to an automorphism of order $\ell$ on $E/H$, and therefore $E/H$ must be isomorphic to $E$.
\end{rem}

We can sum up everything we have discovered in the following theorem:

\begin{theorem}\label{nondisjoint}
Let $E$ be an elliptic curve, and let $D$ be a very ample divisor of degree $n$. Then we have the following:

\begin{enumerate}
\item The 0-dimensional components of $G_{E,\varphi_D}$ are in 1-1 correspondence with Galois subspaces that are disjoint from $\varphi_D(X)$, which are in turn in 1-1 correspondence with the pairs $(H,\langle\xi\rangle,q)$ such that $\xi\in\Aut_0(E)$ is non-trivial, $H\leq E$ is $\xi$-stable, $q\in\varepsilon_{\xi,|H|}^{-1}(\mathcal{S}(D))$ and $n=|H||\xi|$. In this case there are
\[\frac{n}{2}\psi_{2}\left(E,\frac{n}{2}\right)+n\psi_{3}\left(E,\frac{n}{3}\right)+2n\psi_{4}\left(E,\frac{n}{4}\right)+6n\psi_{6}\left(E,\frac{n}{6}\right)\]
0-dimensional components. 
\item The components of $G_{E,\varphi_D}$ of dimension greater than 0 are in 1-1 corres\-pondence with pairs $(H,\langle\xi\rangle)$ where $\xi\in\mathrm{Aut}_0(E)$ is non-trivial, $H\leq E$ is $\xi$-stable and $|H||\xi|<n$. Specifically, if $2\leq s<n$, then $G_{E,\varphi_D}$ has exactly 
\[\psi_2\left(\frac{s}{2}\right)+\psi_3\left(\frac{s}{3}\right)+\psi_4\left(\frac{s}{4}\right)+\psi_6\left(\frac{s}{6}\right)\] components of dimension $n-s$ which correspond to Galois subspaces with Galois groups of order $s$. 
\item If $G=G_{H,\xi,q}$ is a Galois group for $\varphi_D$ of order $s<n$, then the component in $G_{E,\varphi_D}$ that corresponds to $G$ is isomorphic to $\frak{d}_{H,\xi}/H$ which is a $\mathbb{P}^{n-|H||\xi|-1}$-bundle over $E/H$.
\end{enumerate}
\end{theorem}

\begin{rem} When the embedding $\varphi$ is not given by a complete linear system, most of this section can be replicated, except that our definition of $\frak{d}_{H,\xi}$ must change. Indeed, if $\varphi:E\hookrightarrow\mathbb{P}^N$ is an embedding given by a linear system $V_1\leq H^0(E,\mathcal{O}_E(D))$ with corresponding subspace $Q_1\subseteq|D|$, and we set $V_{2,q}:=H^0(E,\mathcal{O}_E(D_{H,\xi,q}))^{G_{H,\xi,q}}$ with $Q_{2,q}\subseteq|D_{H,\xi,q}|$ its corresponding subspace, we should consider  
\[\frak{d}_{V_1,H,\xi}:=\{(F,q)\in\mathrm{Sym}^{r}(E)\times E:F+Q_{2,q}\subseteq Q_1\}.\]
The union of the images of these spaces in $\mathbb{G}(n-3,n-1)$ over all $H$ and $\xi$ gives the space $G_{X,\varphi}$, but it is not clear if $\frak{d}_{V_1,H,\xi}$ is actually a projective bundle over $E/H$. It would be interesting to understand these spaces in the future.
\end{rem}

\begin{rem}
Recall that if $E$ is embedded into $\mathbb{P}^N$ by a very ample divisor $D$, then $D$ is of degree $N+1$. Therefore the previous theorem (which is given in terms of $n$, the degree of $D$) corresponds to an embedding of $E$ into $\mathbb{P}^{n-1}$.
\end{rem}

We end this article by explicitly showing the number of components of $G_{E,\varphi_D}$ for low dimensions, using Theorem \ref{appendixtheorem}. Let $\varphi:E\hookrightarrow\mathbb{P}^N$ be an embedding of $E$ into projective space given by a complete linear system. Table \ref{lowdimension} gives us the number of components of $G_{E,\varphi}$ and their dimensions, depending on $N$ and the $j$-invariant of $E$. Note that for $N=2$, the $1$-dimensional component present for all values of $j$ just corresponds to $\varphi(E)$, since projection from any point $q$ of $\varphi(E)$, composed with $\varphi$, corresponds to taking the quotient of $E$ by $z\mapsto -z+\varphi^{-1}(q)$.\\

\begin{center}
\begin{table}
\begin{tabular}{ |c|c| c |c| }
\hline\multicolumn{1}{|c|}{$\mathbf{N=2}$}&\multicolumn{3}{|c|}{Number of components of $G_{E,\varphi}$} \tabularnewline\hline
 Dimension of component & $j\neq0,1728$ & $j=0$ & $j=1728$ \\\hline
 0&0&3&0\\
 1&1&1&1\\\hline
 Total number: &1&4&1\\
 \hline\hline

\multicolumn{1}{|c|}{$\mathbf{N=3}$}&\multicolumn{3}{|c|}{Number of components of $G_{E,\varphi}$} \tabularnewline\hline
     Dimension of component & $j\neq0,1728$ & $j=0$ & $j=1728$ \\\hline
     0&6&6&14\\
     1&0&1&0\\
     2&1&1&1\\\hline
      Total number:& 7&8 &15  \\
      \hline\hline
      
      \multicolumn{1}{|c|}{$\mathbf{N=4}$}&\multicolumn{3}{|c|}{Number of components of $G_{E,\varphi}$} \tabularnewline\hline
     Dimension of component & $j\neq0,1728$ & $j=0$ & $j=1728$ \\\hline
     0&0&0&0\\
     1&3&3&4\\
     2&0&1&0\\
     3&1&1&1\\\hline
    Total number:&  4& 5&5 \\
      \hline
      \hline
  
      \multicolumn{1}{|c|}{$\mathbf{N=5}$}&\multicolumn{3}{|c|}{Number of components of $G_{E,\varphi}$} \tabularnewline\hline
     Dimension of component & $j\neq0,1728$ & $j=0$ & $j=1728$ \\\hline
     0&12&48&12\\
     1&0&0&0\\
     2&3&3&4\\
     3&0&1&0\\
     4&1&1&1\\\hline
      Total number:&16  & 53& 17 \\\hline
 
\end{tabular}
\caption{Number of components for low dimension}\label{lowdimension}
\end{table}

\end{center}

\newpage

\begin{appendix}

\section{Counting finite subgroups of an elliptic curve}\label{calculate}

The purpose of this appendix is to prove the following theorem:

\begin{theorem}\label{appendixtheorem} Let $E$ be an elliptic curve with an automorphism $\xi$ of order $\ell$ that fixes the origin, and for an integer $m$, let $\psi_\ell(m)$ denote the number of $\xi$-stable subgroups of $E$ of order $m$. We have the following:
\begin{enumerate}
\item $\psi_\ell$ is multiplicative; that is, if $a,b\in\mathbb{N}$ are coprime then $\psi_\ell(ab)=\psi_\ell(a)\psi_\ell(b)$.
\item If $\ell=2$, then $\psi_2(m)=\sigma(m)$ where $\sigma$ is the function that adds all positive divisors of a natural number.
\item\label{ell=3} If $\ell\in\{3,6\}$, $p$ is prime and $\alpha\geq1$, then 
\begin{enumerate}
\item $\psi_\ell(3^\alpha)=1$.
\item $\psi_\ell(p^\alpha)=\alpha+1$ if $p\equiv1\mod{3}$.
\item $\psi_\ell(p^\alpha)=2\lfloor\frac{\alpha}{2}\rfloor-\alpha+1$ if $p\equiv 2\mod{3}$.
\end{enumerate}
\item\label{ell=4} If $\ell=4$, $p$ is prime and $\alpha\geq1$, then
\begin{enumerate}
\item $\psi_4(2^\alpha)=1$.
\item $\psi_4(p^\alpha)=\alpha+1$ if $p^\alpha\equiv 1\mod{4}$.
\item $\psi_4(p^\alpha)=0$ if $p^\alpha\equiv3\mod{4}$.
\end{enumerate}
\end{enumerate}
\end{theorem}

We will start by proving that $\psi_\ell$ is multiplicative. Let $a,b\in\mathbb{N}$ be coprime. Then $\psi_\ell(ab)$ counts the number of $\xi$-stable subgroups of $E[ab]\simeq(\mathbb{Z}/(ab)\mathbb{Z})^2$ of order $ab$. Now the Chinese Remainder Theorem gives us a $\xi$-equivariant isomorphism from $E[ab]$ to $E[a]\times E[b]$, and any $\xi$-invariant subgroup of $E[a]\times E[b]$ of order $ab$ must be the product of a $\xi$-invariant subgroup of $E[a]$ of order $a$ with a $\xi$-invariant subgroup of $E[b]$ of order $b$. This proves that there are exactly $\psi_\ell(a)\psi_\ell(b)$ such subgroups.

This implies that we can restrict our attention to studying subgroups of 
\[E[p^\alpha]\simeq(\mathbb{Z}/p^\alpha\mathbb{Z})^2\]
of order $p^\alpha$ for $p$ a prime number. We will use the classification of subgroups of $(\mathbb{Z}/p^\alpha\mathbb{Z})^2$ found in the proof of \cite[Theorem 3.3]{Tarnauceanu}. Indeed, a subgroup of order $p^\alpha$ must be of the form
\[S_{x,i}:=\langle(p^i,x),(0,p^{\alpha-i})\rangle\]
where $0\leq x< p^{\alpha-i}$ and $0\leq i\leq \alpha$.

In what follows we will analyze the different cases depending on $\ell$.

\subsection{The case $\ell=2$} We note that every subgroup of $E[m]$ is stable by $-1$, and so $\psi_{2}(m)$ is just the number of subgroups of $E[m]$ of order $m$. In \cite[Theorem 3.3]{Tarnauceanu}, if $m=p^\alpha$ for a prime $p$, then this number is $(p^{\alpha+1}-1)/(p-1)$. Therefore for $m=p_1^{\alpha_1}\cdots p_r^{\alpha_r}$ we obtan
\[\prod_{i=1}^r\frac{p_i^{\alpha_i+1}-1}{p_i-1}\]
subgroups. It is an elementary fact that this number corresponds to $\sigma(m)$, which is the sum of all positive divisors of $m$.

\subsection{The case $\ell\in\{3,6\}$} We first observe that a subgroup of $E[m]$ stable under an automorphism of order 6 that fixes the origin is also stable under an automorphism of order 3 that fixes the origin, and vice versa. This implies that $\psi_{3}(m)=\psi_{6}(m)$. Therefore we will concentrate on the case $\ell=3$ in what follows.

Here $\xi$ acts on $E[p^\alpha]\simeq (\mathbb{Z}/p^\alpha\mathbb{Z})^2$ as (a conjugate of) the matrix
\[M_\xi:=\left(\begin{array}{cc}0&-1\\1&-1\end{array}\right).\]

We will first classify what subgroups of $(\mathbb{Z}/p^\alpha\mathbb{Z})^2$ are $M_\xi$-stable, and then we will proceed to count them.

\begin{lemma}\label{characterization for 3}
A subgroup $S_{x,i}\leq(\mathbb{Z}/p^\alpha\mathbb{Z})^2$ is $M_\xi$-stable if and only if $2i\leq\alpha$ and $x\equiv p^{i}y$ for some $y$ such that $p^{\alpha-2i}\mid y^2-y+1$.
\end{lemma}
\begin{proof} In order for $M_\xi S_{x,i}=S_{x,i}$, we need for there to exist $a,b,c,d\in\mathbb{Z}/p^\alpha\mathbb{Z}$ such that $M_\xi(p^i,x)\equiv a(p^i,x)+b(0,p^{\alpha-i})$ and $M_\xi(0,p^{\alpha-i})\equiv c(p^i,x)+d(0,p^{\alpha-i})$. In other words,
\begin{eqnarray*}
(-x,p^i-x)&\equiv&a(p^i,x)+b(0,p^{\alpha-i})\\
(-p^{\alpha-i},-p^{\alpha-i})&\equiv&c(p^i,x)+d(0,p^{\alpha-i})
\end{eqnarray*}

We note that all the congruences above are modulo $p^\alpha$. The first coordinate of the first equation already gives that for $a$ to exist, $x$ must be divisible by $p^i$, and so we write $x=p^iy$. Moreover, $a=-y+tp^{\alpha-i}$ for some $t$. The first coordinate of the second equation immediately gives the condition $2i\leq\alpha$ for $c$ to exist. The second coordinate of the first equation, when expanded, gives
\[p^i(1-y)=-y^2p^i+bp^{\alpha-i}.\]
Therefore, for $b$ to exist we must have that $p^{\alpha-2i}\mid y^2-y+1$. As for $d$, putting $d=y-1$ serves the purpose.
\end{proof}

If $2i\leq\alpha$, we have the natural surjective homomorphism
\[\psi_{p,\alpha,i}:\mathbb{Z}/p^{\alpha}\mathbb{Z}\to\mathbb{Z}/p^{\alpha-2i}\mathbb{Z},\]
and we see that the $y$ from the previous lemma lies in 
\[\psi_{p,\alpha,i}^{-1}(\{z^2-z+1\equiv0\}).\]
Since $x<p^{\alpha-i}$, we must have that $y<p^{\alpha-2i}$. Moreover, since two preimages $y_1,y_2\in\psi_{p,\alpha,i}^{-1}(z)$ differ by a multiple of $p^{\alpha-2i}$, we note that the number of possible values of $x$ is equal to the cardinality of the set $\{z^2-z+1\equiv0\}$.

\begin{lemma}\label{solutions}
Let $\beta\geq1$ be a positive integer and $p$ a prime. Then the equation $z^2-z+1\equiv 0$ has the following number of solutions in $\mathbb{Z}/p^\beta\mathbb{Z}$:
\begin{enumerate}
\item $0$ if $p=3$ and $\beta>1$ or $p\equiv 2\mod{3}$
\item $1$ if $p=3$ and $\beta=1$
\item $2$ if $p\equiv 1\mod{3}$.
\end{enumerate}
\end{lemma}
\begin{proof}
By doing the change of variable $w=-z$, we need to find the number of solutions of $P(w):=w^2+w+1$ in $\mathbb{Z}/p^\beta\mathbb{Z}$. It is trivial to check that for $p=3$ and $\beta=1$ there is only one solution. 

We have that $w^3-1=(w-1)(w^2+w+1)$, and so any solution of $P(w)$ will be a cubed root of 1. If $p=3$ and $\beta>1$, we have that $w^3-1=(w-1)^3$ and so any root of $P(w)$ will be of the form $1+p^\gamma k$ where $k$ is not divisible by $p$. However a quick check shows that $P(w)$ does not have roots of this form. 

We know that $(\mathbb{Z}/p^\beta\mathbb{Z})^\times$ is of order $p^{\beta-1}(p-1)$ and so if $p\neq 3$, for there to be an element of order 3, we need $p\equiv 1\mod{3}$. In particular, in this case $(\mathbb{Z}/p^\beta\mathbb{Z})^\times$ is cyclic and there are then $2$ different elements of order 3. A simple verification shows that neither of these are congruent to $1$ modulo $p$ and so they must be roots of $P(w)$.
\end{proof}

In order to count the number of $M_\xi$-stable subgroups (i.e. $\psi_3(p^\alpha)$), we need to count how many possible $x$'s there are for a given $i$. By the previous analysis, this is just the number of solutions to the equation $z^2-z+1$ in $\mathbb{Z}/p^{\alpha-2i}\mathbb{Z}$, and by the previous lemma we can now count this number explicitly. We will now proceed to prove Part \ref{ell=3} of Theorem \ref{appendixtheorem}.

For $p=3$, if $\alpha$ is even, then for every value of $i<\frac{\alpha}{2}$ we have that $x$ does not exist. For $i=\frac{\alpha}{2}$, we have that there is exactly one possible value for $x$. If $\alpha$ is odd, then the only value for $i$ where an $x$ exists is $i=\frac{\alpha-1}{2}$, and in this case there is one solution as well.

For $p\equiv 1\mod{3}$, if $\alpha$ is even, then for $i<\frac{\alpha}{2}$ we have exactly two possibilities for $x$. If $i=\frac{\alpha}{2}$ then there is only one, and so in total we have $\alpha+1$ values of $x$ (remember that $i$ can be 0). If $\alpha$ is odd, then $0\leq i\leq \frac{\alpha-1}{2}$, and for each of these values, we have two possibilites for $x$. Therefore, there are in total $\alpha+1$ subgroups.

For $p\equiv 2\mod{3}$, we see that the only choice for $i$ where there is a possible $x$ is when $\alpha$ is even and $i=\frac{\alpha}{2}$. In this case there is one possibility. If $\alpha$ is odd then there are none. Therefore, there are exactly $2\lfloor\frac{\alpha}{2}\rfloor-\alpha+1$ subgroups.

\subsection{The case $\ell=4$} Here $\xi$ is a primitive fourth root of unity, and acts on $E[p^\alpha]\simeq(\mathbb{Z}/p^\alpha\mathbb{Z})^2$ as (a conjugate of) the matrix
\[N_\xi:=\left(\begin{array}{cc}0&-1\\1&0\end{array}\right).\]

\begin{lemma}
A subgroup $S_{x,i}\leq(\mathbb{Z}/p^\alpha\mathbb{Z})^2$ is $N_\xi$-stable if and only if $2i\leq\alpha$ and $x\equiv p^{i}y$ for some $y$ such that $p^{\alpha-2i}\mid y^2+1$.
\end{lemma}
\begin{proof}
The proof is analogous to the proof of Lemma \ref{characterization for 3}.
\end{proof}

As in the previous subsection, we therefore need to count the number of solutions to the equation $z^2+1\equiv0\mod{p^{\alpha-2i}}$.

Now it is well-known by using quadratic reciprocity that the equation $z^2+1\equiv0\mod{p^\beta}$ has a solution if and only if either $p=2$ and $\beta=1$ or $p^\beta\equiv 1\mod{4}$. Moreover if $\beta>1$ and there is a solution, then there are exactly 2 solutions. We can now count how many subgroups there are and prove Part \ref{ell=4} of Theorem \ref{appendixtheorem}.

Let us look at the case $p\equiv 3\mod{4}$. If $\alpha$ is even, then for every $0\leq i<\frac{\alpha}{2}$ we have that $p^{\alpha-2i}\equiv1\mod{4}$, and so we have two choices for $x$. Therefore, by counting the case $i=\frac{\alpha}{2}$, we have that there are exactly $\alpha+1$ subgroups. If $\alpha$ is odd, then the equation $z^2+1\equiv 0$ has no solution modulo $p^{\alpha-2i}$ since $p^{\alpha-2i}\equiv 3\mod{4}$. Therefore there are no subgroups. This means that if $p\equiv 1\mod{4}$ then there are $\alpha+1$ subgroups, if $p\equiv 3\mod{4}$ and $\alpha$ is even then there are also $\alpha+1$ subgroups, and if $\alpha$ is odd then there are none. Another way of saying this is that if $p^\alpha\equiv 1\mod{4}$ then there are $\alpha+1$ subgroups and $0$ if $p^\alpha\equiv 3\mod{4}$.

This completes the proof of Theorem \ref{appendixtheorem}.

\end{appendix}

\end{document}